\documentclass[12pt]{amsart}
\usepackage{latexsym,amsfonts,amssymb,epsfig,verbatim}
\usepackage{amsmath, latexsym, graphics, textcomp}
\usepackage{amsthm}
\usepackage{pinlabel}
\usepackage[all]{xy}
\usepackage{graphicx}

\newcommand{\nc}{\newcommand}
\nc{\dmo}{\DeclareMathOperator}
\nc{\nt}{\newtheorem}

\nt{theorem}{Theorem}[section]
\nt{lemma}[theorem]{Lemma}
\nt{proposition}[theorem]{Proposition}
\nt{question}[theorem]{Question}
\nt{corollary}[theorem]{Corollary}
\nt{fact}[theorem]{Fact}
\nt{problem}[theorem]{Problem}
\nt{conjecture}[theorem]{Conjecture}
\dmo{\Mod}{Mod}
\dmo{\SMod}{SMod}
\dmo{\Deck}{Deck}
\dmo{\LMod}{LMod}
\dmo{\Homeo}{Homeo}
\dmo{\color}{color}
\dmo{\perm}{perm}
\dmo{\supp}{supp}
\nc{\p}[1]{\medskip\noindent {\bf #1.}}

\title{Symmetry, Isotopy, and Irregular Covers}
\author{Rebecca R. Winarski}

\begin{document}
\begin{abstract}
We say that a cover of surfaces $S\rightarrow X$ has the Birman--Hilden property if the subgroup of the mapping class group of $X$ consisting of mapping classes that have representatives that lift to $S$ embeds in the mapping class group of $S$ modulo the group of deck transformations.  We identify one necessary condition and one sufficient condition for when a cover has this property.  We give new explicit examples of irregular branched covers that do not satisfy the necessary condition as well as explicit covers that satisfy the sufficient condition.  Our criteria are conditions on simple closed curves, and our proofs use the combinatorial topology of curves on surfaces.\end{abstract}
\maketitle
\bibliographystyle{plain}

\section{Introduction}
A general problem is to understand all (injective) homomorphisms between (finite index subgroups of) mapping class groups of surfaces, see e.g. \cite{AS,ALS,HK,IM,PB}.
Birman and Hilden proved that if $\widetilde{X}\rightarrow X$ is a regular branched covering space of surfaces, there is an embedding of a finite-index subgroup of the mapping class group of $X$ to a finite quotient of the mapping class group of $\widetilde{X}$.   The goal of this paper is to extend their results to certain irregular covers and describe a class of covers where their results do not extend.

Let $S$ be a surface, by which we mean the space obtained from a closed, oriented 2-manifold by removing finitely many points and choosing finitely many points to be distinguished, or {\it marked}.  The mapping class group of $S$, denoted $\Mod(S)$, is the group of diffeomorphisms of $S$ that leave the set of marked points invariant, up to isotopies that leave the set of marked points invariant.

Throughout the paper we assume all covering spaces of surfaces are connected, oriented and possibly branched.  Fix such a covering space $p:\widetilde{X}\rightarrow X$.  If the cover is branched, we regard the branch points as marked points in $X$.  In fact, we will assume that the only marked points in $X$ are the branch points.  In particular, the mapping class group of $X$ must fix the set of branch points.  Let $\LMod(X)$ be the finite-index subgroup of $\Mod(X)$ consisting of the isotopy classes of the diffeomorphisms of $X$ that lift to diffeomorphisms of $\widetilde{X}$.  Let $\SMod(\widetilde{X})$ be the subgroup of $\Mod(\widetilde{X})$ consisting of the isotopy classes of the fiber preserving, or {\it symmetric}, diffeomorphisms of $\widetilde{X}$.  Note that by isotopy, we mean arbitrary isotopy, not necessarily isotopy through symmetric diffeomorphisms.  

When $\chi(\widetilde{X})<0$, the group of deck transformations $\Deck(p)$ maps injectively to a subgroup of $\SMod(\widetilde{X})$ \cite[Proposition 7.7]{primer}.  Since isotopies of $X$ lift to isotopies of $\widetilde{X}$, there is a natural surjective homomorphism $\Phi:\LMod(X)\rightarrow \SMod(\widetilde{X})/\Deck(p)$.  
We say that $p:\widetilde{X}\rightarrow X$ has the {\it Birman--Hilden property} if the map $\Phi$ is an isomorphism.

In this paper, we investigate the question of which covers have the Birman--Hilden property.  The Birman--Hilden property has been studied in various cases:
\begin{itemize}
\item {\it Regular covers:} Birman and Hilden proved that regular covers have the Birman--Hilden property \cite{BH1,BH}.   MacLachlan and Harvey gave a new proof using Teichm\"uller theory \cite{MH}.  
\item  {\it Unbranched covers:} Aramayona, Leininger, and Souto proved that unbranched covers have the Birman--Hilden property \cite[Proposition 5]{ALS}.
\item {\it Irregular branched covers:}  It is easy to see that the 3-fold irregular branched cover of a surface of genus $g$ over a sphere with $2g+4$ branch points does not have the Birman--Hilden property \cite{BW,fuller}, and we explain this in Section \ref{WCL}.  In contrast, Aramayona, Leininger, and Souto found an example of an irregular branched cover that does have the Birman--Hilden property \cite{ALS}.
\end{itemize}

In summary, the Birman--Hilden property is understood for all covers except irregular branched covers where there are both examples that have the Birman--Hilden property and examples that do not have the Birman--Hilden property.  
In this paper we give one new sufficient condition and one new necessary condition for the Birman--Hilden property.   We also introduce new explicit examples that satisfy our sufficient condition and new explicit examples that do not satisfy our necessary condition.

\p{Property NU} We say a (possibly branched) covering space of surfaces $\widetilde{X}\rightarrow X$ has {\it property NU} if there is no unramified preimage in $\widetilde{X}$ of a branch point.
 \begin{theorem}\label{ramified}
Let $p:\widetilde{X}\rightarrow X$ be a finite covering space of surfaces such that $\chi(\widetilde{X})<0$.  If $p$ has property NU, then $p$ has the Birman--Hilden property.
\end{theorem}

Farb and Margalit gave a new proof of the Birman--Hilden property for the regular branched cover induced by the hyperelliptic involution of a surface \cite[Theorem 9.2]{primer}.  We generalize their approach, which translates the Birman--Hilden property, a property of diffeomorphisms, into a property of simple closed curves. 

We prove Theorem \ref{ramified} in Section \ref{maintheorems}, which is the technical heart of the paper.    
Theorem \ref{ramified} recovers the result of Aramayona, Leininger, and Souto \cite{ALS} that unbranched covers have the Birman--Hilden property.  It also recovers the regular case proved by Birman and Hilden.  In addition, Theorem \ref{ramified} makes it easy to verify that some covers have the Birman--Hilden property, where it was previously unknown.  Examples are shown in Section \ref{examples}.

Theorem \ref{ramified} can also be proved using Teichm\"uller theory, following some of the ideas in the work of MacLachlan and Harvey.  We sketch a proof in Section \ref{maintheorems}.

\p{Curve lifting property}   An essential curve is a curve that is not isotopic to a boundary component or a single point.  
We say that a covering space of surfaces has the {\it weak curve lifting property} if the preimage in $\widetilde{X}$ of every essential simple closed curve in $X$ has at least one essential connected component.

 \begin{proposition}\label{necessity}
Let $p:\widetilde{X}\rightarrow X$ be a finite covering space of surfaces such that $\chi(\widetilde{X})<0$.
If $\widetilde{X}\rightarrow X$ has the Birman--Hilden property, then it has the weak curve lifting property. \end{proposition}

We prove Proposition \ref{necessity} in Section \ref{WCL}.  

\medskip
The weak curve lifting property can be checked algorithmically.  Theorem \ref{checkingcl} gives an algorithm for choosing a finite set of simple closed curves sufficient for checking the weak curve lifting property.

\medskip
As a consequence of Proposition \ref{necessity}, we have the following theorem.

\begin{theorem}\label{simple}
Let $p:\widetilde{X}\rightarrow X$ be a $n$-fold covering space of surfaces such that $\chi(\widetilde{X})<0$ and the complement of the branch points in $X$ is not a sphere with three punctures. Suppose also that $n\geq 3$ and that $X$ has at least two branch points.
If $p$ is simple, then $\widetilde{X}\rightarrow X$ does not have the Birman--Hilden property.
\end{theorem}
We prove Theorem \ref{simple} in Section \ref{Simple}.  
The 3-fold irregular branched cover of a surface of genus $g$ over a sphere with $2g+4$ branch points described above is a simple cover, so Theorem \ref{simple} generalizes this example.

\p{Curve complex} It is well known that the mapping class group of a surface $S$ is the group of automorphisms of the curve complex $\mathcal{C}(S)$ of $S$.  Let $\widetilde{X}\rightarrow X$ be a covering space.  It is natural to ask if there is any relationship between $\mathcal{C}(X)$ and $\mathcal{C}(\widetilde{X})$.  Luo showed that there is an isomorphism between $\mathcal{C}(X)$ and $\mathcal{C}(\widetilde{X})$ in only three cases \cite{luo}.  Rafi and Schleimer showed that there is a quasi-isometric embedding from $\mathcal{C}(X)$ to $\mathcal{C}(\widetilde{X})$ when $\widetilde{X}\rightarrow X$ is an unbranched covering space, but there is not a quasi-isometric embedding from $\mathcal{C}(X)$ to $\mathcal{C}(\widetilde{X})$ in general \cite{RS}.  Further research could pursue necessary and sufficient conditions for the existence of an quasi-isometric embedding from $\mathcal{C}(X)$ to $\mathcal{C}(\widetilde{X})$.  In particular: if $\widetilde{X}\rightarrow X$ has the Birman--Hilden property, is there a quasi-isometric embedding from $\mathcal{C}(X)$ to $\mathcal{C}(\widetilde{X})$?

\p{Ramification number} 
We say that a covering space of surfaces has the {\it equal ramification number property} if for every point $x\in X$, all points of $p^{-1}(x)$ have the same ramification number (see Section \ref{previous} for the definition of ramification number).  Regular covers and unbranched covers have the equal ramification number property.    The equal ramification number property implies Property NU.  Therefore by Theorem \ref{ramified}, covers that have the equal ramification number property have the Birman--Hilden property.
There are some irregular, branched covers that have equal ramification number property.

One might guess that the Birman--Hilden property is equivalent to the equal ramification number property.  However, Aramayona, Leininger, and Souto give an example of a cover that has unramifed preimage and the Birman--Hilden property \cite[Section 4]{ALS}.  

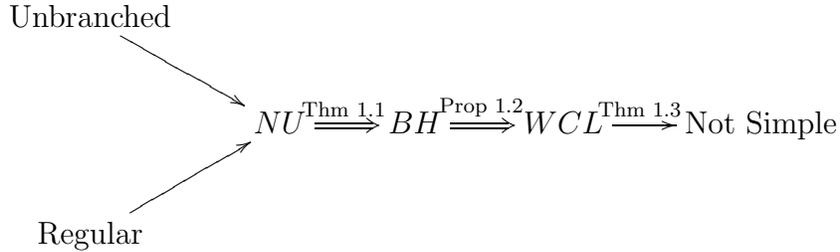
\begin{figure}
\label{diagram}
\begin{displaymath}
\xymatrix{\text{Unbranched}\ar@{->}[rd]&  &\\&NU\ar@{=>}[r]^{\text{Thm \ref{ramified}\ }}&BH\ar@{=>}[r]^{\text{Prop \ref{necessity}\ \ }} &WCL\ar@{->}[r]^{\text{Thm \ref{simple}\ \ \ \ \ \ }}&\text{Not Simple}\\
\text{Regular}\ar@{->}[ru]& &&&}
\end{displaymath}
\caption{Implications of relevent properties of finite covers.  The double arrows represent the main results of the paper.  The single arrows are implications proved in corollaries.}
\end{figure}
\p{Acknowledgments} The author would like to thank her advisor, Dan Margalit, for suggesting the problem and many discussions.  She would also like to thank Ian Agol, Joan Birman, Meredith Casey, John Etnyre, Allen Hatcher, Chris Leininger, and Andrew Putman for helpful comments and conversations.

\section{Branched Covers}\label{previous}
In this section, we give definitions related to branched covers.  We also show a construction of irregular branched covers that have Property NU, hence the Birman--Hilden property.

\subsection{Branched covers} 
Let $X$ be a surface.  A {\it branched cover} over $X$ is a surface $\widetilde{X}$ and a map $p:\widetilde{X}\rightarrow X$ such that for some finite set $B$ in $X$ we have:\begin{enumerate} \item $p|_{\widetilde{X}\setminus p^{-1}(B)}$ is a covering map; \item for each $x\in B$ there exists an open neighborhood $U\subset X$, called a {\it standard neighborhood}, such that \begin{enumerate} \item $U\cap  B=x$\item Each component $V_i$ of $p^{-1}(U)$ is an open disk and contains one point $\widetilde{x}_i\in p^{-1}(x)$ \item Each $p|_{V_i\setminus \widetilde{x}_i}$ is covering map $V_i\setminus \widetilde{x}_i\rightarrow U\setminus x$; we refer to the degree of the cover, $m_i$, as the {\it ramification number} of $\widetilde{x}_i$ \item There is at least one $\widetilde{x}_i$ such that $m_i>1$\end{enumerate}\end{enumerate}  

We say that a preimage $\widetilde{x}_i$ of $x$ is {\it ramified} if its ramification number $m_i$ is greater than 1; otherwise it is {\it unramified}.  Each point in $B$ is called a {\it branch point}.  A branch point may have preimages with different ramification numbers.  

Recall that in a branched covering space of surfaces $p:\widetilde{X}\rightarrow X$, each branch point of $X$ is considered as a marked point.  In particular, $\Mod(X)$ is defined relative to the branch points.
Throughout the paper, let $X^{\circ}$ be the (noncompact) surface that is the complement of the branch points in $X$.  Let $p:\widetilde{X}\rightarrow X$ be a branched covering space of surfaces.  Then we will define $\widetilde{X}^{\circ}=p^{-1}(X^{\circ})$.

A {\it deck transformation} of a branched covering space of surfaces $p:\widetilde{X}\rightarrow X$ is a homeomorphism of $\widetilde{X}$ that restricts to a deck transformation of the covering space $\widetilde{X}^{\circ}\rightarrow X^{\circ}$.  A covering space $\widetilde{X}\rightarrow X$ is {\it regular} if the group of deck transformations acts transitively on the fibers of each point in $X$, and is {\it irregular} otherwise.

\subsection{Simple closed curves and branched covers}
Let $p:\widetilde{X}\rightarrow X$ be a covering space with branch points $B\subset X$.  A {\it simple closed curve} in $X$ is the image of an injective map $S^1\rightarrow X^{\circ}\subseteq X$.  In particular, there is a bijective correspondence between simple closed curves in $X$ and simple closed curves in $X^{\circ}$.  There is not a bijective correspondence between simple closed curves in $\widetilde{X}$ and simple closed curves in $\widetilde{X}^{\circ}$ because simple closed curves in $\widetilde X$ are allowed to contain points of $p^{-1}(B)$.
As a result, curves that are isotopic in $X$ are isotopic in $X^{\circ}$, but curves that are isotopic in $\widetilde{X}$ might not be isotopic in $\widetilde{X}^{\circ}$.

A {\it multicurve} in a surface $S$ is a set of pairwise disjoint simple closed curves $\{\alpha_1,\cdots,\alpha_k\}$ in $S$.  We say that two multicurves $\alpha=\{\alpha_1,\cdots,\alpha_k\}$ and $\beta=\{\beta_1,\cdots,\beta_l\}$ are isotopic if $k=l$, and they are homotopic through $k$-component multicurves.    

\subsection{Constructing irregular, branched covers}\label{examples}
One method of constructing irregular covers is by cutting surfaces and sewing their resulting boundary components as follows.  Let $S_g^b$ denote a surface of genus $g$ with $b$ boundary components.  Figure \ref{sixfold} shows a 6-fold cover of a torus with two branch points by $S^0_4$. 
\begin{figure}[h]
\begin{center}

\labellist\small\hair 2.5pt
        \pinlabel {{\it cut}} by 0 .5 at 120 234
         \pinlabel {{\it rotate}} by 0 0 at 125 137
\pinlabel {$\alpha$} by 0 0 at 150 20
\pinlabel {$\widetilde{\alpha}$} by 0 0 at 48 270
\pinlabel {$\widetilde{\alpha}$} by 0 0 at 92 270
\pinlabel {$\widetilde{\alpha}$} by 0 0 at 138 270
 \pinlabel {{\it sew}} by 0 .5 at 123 66 
        \endlabellist
\includegraphics[scale=1]{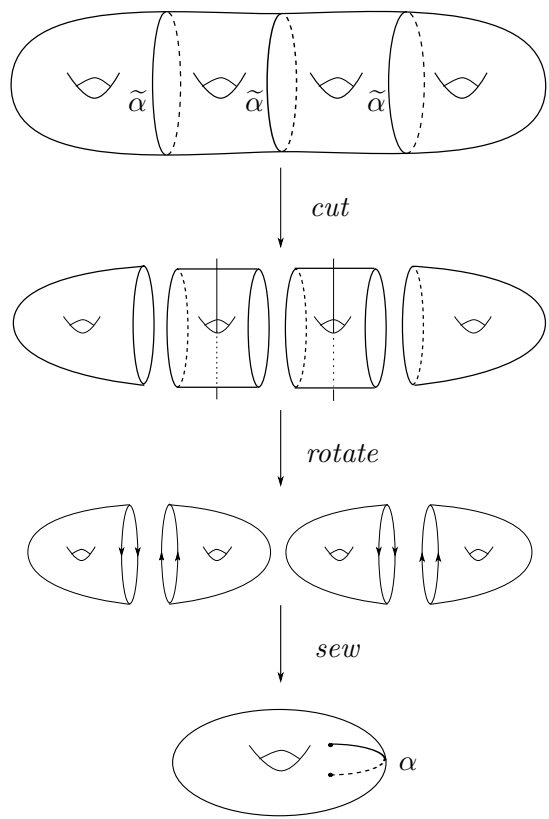}
\end{center}
\caption{The 6-fold branched, irregular covering of the twice marked torus comes from cutting along the curves marked and identifying along boundaries}
\label{sixfold}
\end{figure}
We construct the covering map in 3 steps:
\begin{enumerate}
 \item Cut $S_4^0$.
\item Map each $S_1^d$ to $S_1^1$ by any covering map.
\item Sew the boundary of each $S_1^1$.
\end{enumerate}
The covering space is constructed by cutting $S^0_4$ along the curves $\widetilde{\alpha}$ into four tori with one or two boundary components as shown in Figure \ref{sixfold}.  Then take a quotient by rotation of each $S_1^d$ by $\frac{2\pi}{d}$.  We then have four disjoint copies of $S_1^1$.  There is a map from the four disjoint copies of $S_1^1$ to the torus with two branch points induced by identifying the top and bottom of the boundary component of each $S_1^1$ as indicated by the arrows.  This procedure constructs a well-defined 6-fold covering space of $S^0_4$ over the torus with two branch points.

\begin{figure}[h]
\begin{center}
\includegraphics[scale=1]{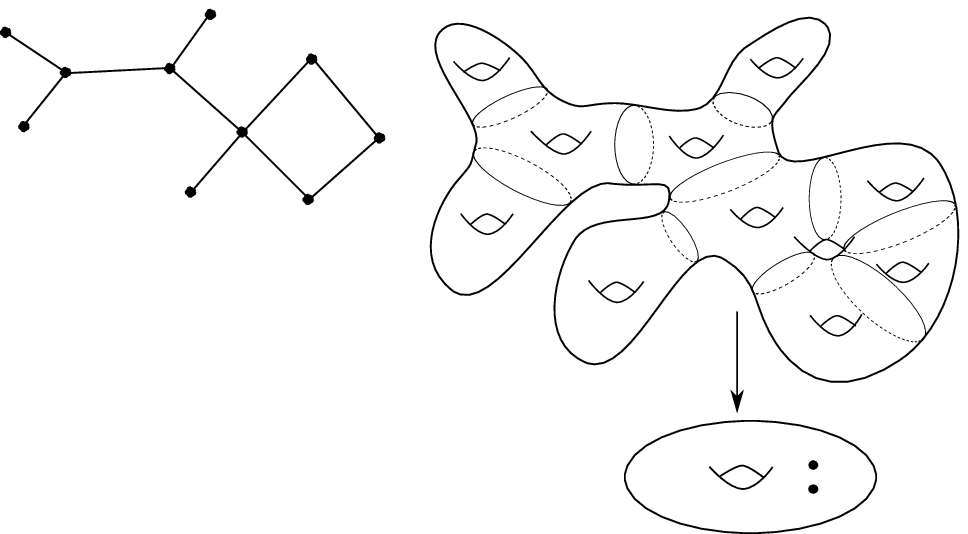}
\end{center}
\caption{Constructing an irregular branched cover from a graph.}
\label{graph}
\end{figure}

All three preimages of each branch point are ramified, so the cover has property NU.  Therefore by Theorem \ref{ramified} the covering space described above has the Birman--Hilden property.

One can generalize the cut-and-identify construction in Figure \ref{sixfold} to make a family of covers of the torus with two branch points as follows.  In fact, we can make a cover of any even degree.

\begin{theorem}
Let $G$ be a finite graph.  Let $X$ be a torus with two branch points.  One may construct a finite covering space of surfaces $\widetilde{X}\rightarrow X$ corresponding to $G$ that has the Birman--Hilden property.
\end{theorem}
\begin{proof}
We first construct $\widetilde{X}$.  For each vertex of $G$ of degree $d$, start with one copy of $S_1^d$.  Then identify the boundaries of two components if and only if there is an edge between the corresponding vertices of $G$.  Figure \ref{graph} shows the surface constructed from a graph.

To create the covering map, we repeat the procedure as above.  First cut $\widetilde{X}$ at each corresponding edge of $G$.  For every integer $d\geq 1$, take the quotient by rotation of $S_1^d$ by $2\pi/d$ of each $S_1^d$ to obtain a copy of $S_1^1$, analogously to the example in Figure \ref{sixfold} above.  Then we have a disjoint union of components each of which is $S_1^1$.  For each torus, identify the top and bottom of the boundary component, creating two fixed points.  We then have a cover of the torus with two branch points by a disjoint union of tori with two marked points.

As above, one can show that $\widetilde{X}\rightarrow X$ is a branched covering space. In particular, all preimages of both branch points have ramification number 2.  Therefore the cover has property NU.

By Theorem \ref{ramified}, $\widetilde{X}\rightarrow X$ has the Birman--Hilden property.
\end{proof}
By choosing $G$ to be a planar graph with no automorphisms, we construct $\widetilde{X}\rightarrow X$ to be an irregular branched covering space of surfaces.

\section{The Weak Curve Lifting property}
\label{WCL}
In this section, we prove Proposition \ref{necessity}, which states that the Birman--Hilden property implies the weak curve lifting property.   Additionally we prove that there is an algorithm to check the weak curve lifting property.
\subsection{Proof of Proposition \ref{necessity}}
First we prove that covers with the Birman--Hilden property have the weak curve lifting property.
\begin{proof}[Proof of Proposition \ref{necessity}]
 Suppose that the cover $p:\widetilde{X}\rightarrow X$ does not have the weak curve lifting property, that is, there exists an essential simple closed curve $\gamma$ in $X$ such that $p^{-1}(\gamma)$ has no essential connected component.  We want to show that $\Phi:\LMod(X)\rightarrow\SMod(\widetilde{X})/\Deck(p)$ has nontrivial kernel.

Let $\varphi:X\rightarrow X$ be a Dehn twist about $\gamma$.  
Since $\LMod(X)$ has finite index in $\Mod(X)$, there is some integer $k$ such that the class of $\varphi^k$ is in $\LMod(X)$.  The isotopy class of $\varphi^k$ is a nontrivial mapping class because it is a power of a Dehn twist about an essential simple closed curve \cite[Proposition 3.1]{primer}.

Since $\varphi^k$ is supported on a neighborhood of $\gamma$, there is a preferred lift $\widetilde{\varphi}^k:\widetilde X\rightarrow\widetilde X$ of $\varphi^k$ that is supported on a neighborhood of the components of $p^{-1}(\gamma)$, each of which is an annulus.  Therefore $\widetilde{\varphi}^k$ is a composition of Dehn twists in these annuli.  Since all components of $p^{-1}(\gamma)$ are inessential, the Dehn twist around each component of $p^{-1}(\gamma)$ is isotopic to $id_{\widetilde{X}}$.  Therefore $\widetilde{\varphi}^k$ is isotopic to $id_{\widetilde{X}}$, and hence its class is the trivial mapping class.  
Since $\varphi^k$ does not represent the trivial mapping class but is in the kernel of $\Phi$, $p$ does not have the Birman--Hilden property.
\end{proof}

\subsection{A covering space without the Birman-Hilden property} 
Consider the 3-fold simple branched covering space of surfaces $p:\widetilde{X}\rightarrow X$ where $\widetilde{X}$ is a closed surface of genus $g$ and $X$ is a sphere with $2g+4$ branch points; (this covering space is unique up to isomorphism).   The example where $g=3$ is shown in Figure \ref{fullerp}.  The preimages of the branch points are shown but, as usual, are not thought of as marked points in $\widetilde{X}$.

\begin{figure}[h]
\begin{center}

\labellist\small\hair 2.5pt
        \pinlabel {$X$} by 1 0 at 230 50
        \pinlabel {$\gamma$} by -1 1 at 150 80
          \pinlabel $\gamma_2$ by 0 1 at 250 270 
         \pinlabel $\gamma_3$ at 250 290 
        \pinlabel $\widetilde X$ [lb] at 270 275
        \pinlabel $\gamma_1$ by 0 0 at 242 160
        \endlabellist
\centerline{\includegraphics[width=3in]{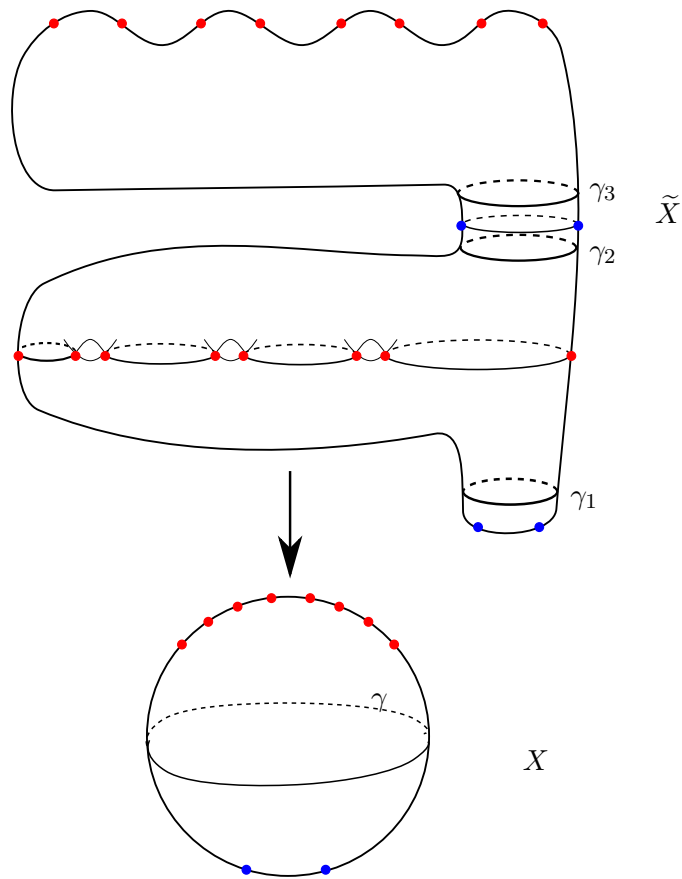}}
\end{center}
\caption{An irregular, branched covering space that does not have the Birman--Hilden property.}

\label{fullerp}\end{figure}

We can choose $p$ so that the preimage of the simple closed curve $\gamma$ in Figure \ref{fullerp} is the union of $\widetilde{\gamma}_1,\widetilde{\gamma}_2,$ and $\widetilde{\gamma}_3$ in $\widetilde{X}$.  Since each $\widetilde{\gamma}_i$ is inessential, the covering space $p$ does not have the weak curve lifting property, and therefore does not have the Birman--Hilden property by Proposition \ref{necessity}.

The covering space $p$ is one example of a simple cover that does not have the Birman--Hilden property.  See Fuller's paper \cite{fuller} for more discussion of this example.  Theorem \ref{simple}, proven in Section \ref{Simple} below, says that no simple covers with degree at least 3 and at least two branch points have the weak curve lifting property.  Therefore no such covers have the Birman--Hilden property.

\subsection{Checking the weak curve lifting property}\label{algorithm}
The usefulness of the weak curve lifting property lies in the fact that given a covering space $\widetilde{X}\rightarrow X$, one only needs to check a finite number of curves in $X$ in order to determine whether or not it holds.  Our next goal is to show that there is an algorithm to select a sufficient set of simple closed curves, as we will prove in Theorem \ref{checkingcl} below.  First we need a proposition for unbranched covers.  

\begin{proposition}\label{cosets}
Let $p:\widetilde{X}\rightarrow X$ be an $n$-fold unbranched covering space of surfaces.  
There is an algorithm to find coset representatives of $\LMod(X)$ in $\Mod(X).$
\end{proposition}
\begin{proof}
Fix a base point $x_0$ in $X$ and and consider it as a marked point, and choose $\widetilde{x}_0\in p^{-1}(x_0)$.  The mapping class group relative to the base point, denoted $\Mod(X,x_0)$, is the group of isotopy classes of diffeomorphisms of $X$ that fix $x_0$.  
We define $\LMod(X,x_0)$ to be the subgroup of $\Mod(X,x_0)$ consisting of mapping classes that have representatives that lift to $\Mod(\widetilde{X},\widetilde{x}_0)$.

Let $H_0=p_*(\pi_1(\widetilde{X},\widetilde{x}_0))$.  Then $H_0$ is an index $n$ subgroup of $\pi_1(X,x_0)$.  There is an action of $\Mod(X,x_0)$ on the set of index $n$ subgroups of $\pi_1(X,x_0)$; denote the orbit of $H_0$ by $\mathcal{H}$. The action of $\Mod(X,x_0)$ on the subgroups of $\pi_1(X,x_0)$ induces a map $P:\Mod(X,x_0)\rightarrow\Sigma_{\mathcal{H}}$.

We have two facts:
\begin{enumerate}
\item $\LMod(X,x_0)$ is the stabilizer of $H_0$ in $\Mod(X,x_0)$.
\item $\mathcal{H}$ is finite.
\end{enumerate}
The first follows from the lifting criterion for covering spaces and the fact that all isotopies of $X$ lift to isotopies of $\widetilde{X}$.  The second follows from the fact that there are finitely many index $n$ subgroups of any finitely generated group. 

To find coset representatives of $\LMod(X, x_0)$, it suffices to find a finite set $F\subset\Mod(X,x_0)$ such that for all $H\in\mathcal{H}$ there exists $f\in F$ such that $f\cdot H_0=H$.

Let $\{g_i\}$ be any finite generating set of $\Mod(X,x_0)$.  The set $\{P(g_i)\}$ generates a subgroup of $\Sigma_{\mathcal{H}}$.  Because $\Sigma_{\mathcal{H}}$ is finite, there is an element in Im$(P)$ whose word length $L$ with respect to $\{P(g_i)\}$ is maximal.
There is a finite set of elements $F$ in $\Mod(X,x_0)$ of word length at most $L$ with respect to $\{g_i\}$.  

By construction $P(F)=$ Im$(P)$.  In other words, any permutation of elements of $\mathcal{H}$ obtained by the action of $\Mod(X,x_0)$ has a preimage under $P$ in $F$.  Therefore, for all $H\in\mathcal{H}$, there is $f\in F$ such that $f\cdot H_0=H$.

The forgetful homomorphism $\mathcal{F}:\Mod(X,x_0)\rightarrow\Mod(X)$ is surjective.  Additionally, any diffeomorphism of $X$ that lifts to $\widetilde{X}$ can be modified by isotopy to fix $x_0$.  Since every isotopy of $X$ lifts to an isotopy of $\widetilde{X}$, the restriction of $\mathcal{F}$ to $\LMod(X,x_0)$ surjects onto $\LMod(X)$.  The image $\{\mathcal{F}(f_j)\}$ of the coset representatives of $\LMod(X,x_0)$ in $\Mod(X,x_0)$ is a set of coset representatives of the subgroup $\LMod(X)$ in $\Mod(X)$.
\end{proof}

We recall the notation that $X^{\circ}$ is the surface $X$ with the branch points removed and $\widetilde{X}^{\circ}=p^{-1}(X^{\circ})$.  Denote the set of isotopy classes of simple closed curves in a surface $S$ by $\mathcal{G}(S)$. Denote the set of isotopy classes of multicurves in $S$ by $\mathcal{M}(S)$. 

\begin{theorem}\label{checkingcl}
Let $p:\widetilde{X}\rightarrow X$ be an $n$-fold covering space of surfaces such that $\chi(\widetilde{X})<0$.  There is an algorithm to check whether or not $p$ has the weak curve lifting property.\end{theorem}
\begin{proof}
The embedding $X^{\circ}\rightarrow X$ takes simple closed curves in $X^{\circ}$ to simple closed curves in $X$.   Then there is a sequence of maps as follows; we denote their composition by $\Psi$:
\begin{displaymath}
\Psi:\xymatrix{\mathcal{G}(X)\ar@{->}[r]^{\cong}&\mathcal{G}(X^{\circ})\ar@{->}[r]^{\text{preimage}}&\mathcal{M}(\widetilde{X}^{\circ})\ar@{->}[r]^{\text{fill in}} &\mathcal{M}(\widetilde{X})}
\end{displaymath}

The second map takes a simple closed curve in $X^{\circ}$ to its preimage in $\widetilde{X}^{\circ}$, and the third map takes multicurves in the one induced by the inclusion $\widetilde{X}^{\circ}\rightarrow\widetilde{X}$.

If $c,c'\in\mathcal{G}(X)$ differ only by an element of $\LMod(X)$, then $\Psi(c)$ and $\Psi(c')$ have the same number of essential and inessential components.  Therefore to check the weak curve lifting property, it is sufficient to compute $\Psi(c)$ for a set of representatives of $\mathcal{G}(X)/\LMod(X)$.

The set $\mathcal{G}(X)/\Mod(X)$ is finite \cite[Section 1.3]{primer}.  Let $\{\gamma_i\}$ be the set of representatives of the quotient.

By Proposition \ref{cosets}, we can choose a finite set of coset representatives $\{h_j\}$ of $\LMod(X^{\circ})$ in $\Mod(X^{\circ})$. 
Therefore it is sufficient to check the weak curve lifting property by calculating the image under $\Psi$ of the isotopy classes of simple closed curves in the set $\{h_j(\gamma_i)\}$.
\end{proof}

\section{Proof of Theorem \ref{ramified}}\label{main}
In this section we will prove Theorem \ref{ramified}.  As an intermediate step we define the isotopy projection property and prove it implies the Birman--Hilden property.
\subsection{Isotopy Projection Property}\label{isotopyprojection}
We say a covering space of surfaces $p:\widetilde{X}\rightarrow X$ has the {\it isotopy projection property} if the following holds: 
\begin{quote} For any pair of simple closed curves $\alpha$ and $\beta$ in $X$ such that $p^{-1}(\alpha)$ and $p^{-1}(\beta)$ are isotopic multicurves, we have that $\alpha$ and $\beta$ are isotopic.\end{quote}

Recall that branch points in $X$ are marked points, but their preimages in $\widetilde{X}$ are not marked.  Therefore homotopies in $\widetilde{X}$ may not project to homotopies in $X$ that fix the set of marked points.

We will use the fact that the isotopy projection property implies the Birman--Hilden property to prove Theorem \ref{ramified}.
\begin{proposition}\label{curveisotopy}
 Let $p:\widetilde{X}\rightarrow X$ be a finite covering space of surfaces where $\chi(\widetilde{X})<0$.  The covering space $p$ has the isotopy projection property if and only if it has the Birman--Hilden property.
\end{proposition}

Below we prove the forward direction of Proposition \ref{curveisotopy}, namely, that the isotopy projection property implies the Birman--Hilden property.  The other direction is left as an exercise, and is not used in this paper.  

\p{Bigons} A {\it bigon} is a cell structure on a disk with one 2-cell, two 1-cells, and two 0-cells.  A bigon in a surface $X$ is an embedding of the cell complex into $X$.  We call the 1-cells {\it edges} of the bigon and the 0-cells {\it vertices}.  We say a bigon is bounded by two curves $\alpha$ and $\beta$ if the edges of the bigon are contained in $\alpha$ and $\beta$ respectively.  Let $B$ be a bigon in $X$ bounded by $\alpha$ and $\beta$.  We let $B_{\alpha}$ and $B_{\beta}$ be the edges of $B$ contained in $\alpha$ and $\beta$.  An {\it innermost bigon} is a bigon bounded by $\alpha$ and $\beta$ that does not contain any other bigons bounded by $\alpha$ and $\beta$.

\p{Finite order elements} Before completing the proof of Proposition \ref{curveisotopy}, we prove the following lemma about finite order elements.
\begin{proposition}\label{tf}
  Let $p:\widetilde{X}\rightarrow X$ be a finite covering space of surfaces where $\chi(\widetilde{X})<0$.  If $p$ has the isotopy projection property, then the kernel of $\Phi:\LMod(X)\rightarrow \SMod(\widetilde{X})/\Deck(p)$ is torsion free.
\end{proposition}
\begin{proof}
Let $f$ be a mapping class of order $k$ in the kernel of $\Phi$.  By a classical theorem of Fenchel and Nielsen \cite[Theorem 7.1]{primer}, there is a diffeomorphism $\varphi$ of order $k$ that represents $f$.  There is a lift $\widetilde{\phi}:\widetilde{X}\rightarrow\widetilde{X}$ of the diffeomorphism $\varphi$ that is isotopic to the identity in $\widetilde{X}$.  Then $\varphi^k$ is the identity on $X$ and lifts to $\widetilde{\varphi}^k$, which is a deck transformation.  Thus $\widetilde{\varphi}$ is a finite-order diffeomorphism of $\widetilde{X}$.  Since $\chi(\widetilde{X})<0$, the only finite-order diffeomorphism of $\widetilde{X}$ that is isotopic to $id_{\widetilde{X}}$ is $id_{\widetilde{X}}$ itself.  Therefore, $\widetilde{\varphi}$ is exactly $id_{\widetilde{X}}$, so $\varphi$ is exactly $id_X$.
\end{proof}

\p{Finishing the proof} We are now ready to prove that the isotopy projection property implies the Birman--Hilden property.  
\begin{theorem}[Casson--Bleiler]\label{cb}
 Let $S$ be a surface, with $\chi(S)<0$.  Let $\phi$ be a diffeomorphism of $S$.  If for every essential simple closed curve $\gamma$, there is some $k>0$ such that $\phi^k(\gamma)$ is isotopic to $\gamma$, then $\phi$ is finite order.
\end{theorem}
A proof is given in \cite{cassonbleiler}.

\begin{proof}[Proof of Proposition \ref{curveisotopy}]
Suppose $p:\widetilde{X}\rightarrow X$ has the isotopy projection property.  We want to show that the map $\Phi:\LMod(X)\rightarrow \SMod(\widetilde{X})/\Deck(p)$ is injective.

Fix an essential simple closed curve $\gamma$ in $X$.  Let $\widetilde{\gamma}=p^{-1}(\gamma)$ in $\widetilde{X}$.

Let $\varphi$ be a diffeomorphism of $X$ whose class is in ker$\Phi$.  We need to show that $\varphi$ is isotopic to $id_X$.

Since $\widetilde{\varphi}$ is isotopic to $id_{\widetilde{X}}$, the multicurve $\widetilde{\varphi}(\widetilde{\gamma})$ is isotopic to $\widetilde{\gamma}$. Because $p$ has the isotopy projection property, $\gamma$ is isotopic to $\varphi(\gamma)$.  Since this is true for any essential simple closed curve $\gamma$, Theorem \ref{cb} implies $\varphi$ is finite order.

By Proposition \ref{tf}, $\varphi$ is isotopic to the identity.
\end{proof}

\subsection{Property NU}\label{maintheorems}
Now we prove Theorem \ref{ramified}. We recall that a cover has property NU if it has no unramified preimages of branch points.  Theorem \ref{ramified} states that covers with NU have the Birman--Hilden property.

To prove Theorem \ref{ramified}, we prove that coverings spaces of surfaces with property NU have the isotopy projection property.  

We fix the following notation: $p:\widetilde{X}\rightarrow X$ is a finite branched covering space of surfaces where $\chi(\widetilde{X})<0$, $\alpha$ and $\beta$ are transverse simple closed curves in $X$, and $\widetilde{\alpha}$ and $\widetilde{\beta}$ are the multicurves $p^{-1}(\alpha)$ and $p^{-1}(\beta)$ in $\widetilde{X}$. Also, $\bar{X}$ is the surface obtained from $X$ by forgetting that the branch points are marked.

Before we prove Theorem \ref{ramified}, we prove Lemmas \ref{bigons} and \ref{annuli}, which roughly state that bigons in $\widetilde{X}$ project to bigons in $X$ and disjoint isotopic curves in $\widetilde{X}$ project to disjoint isotopic curves in $X$.

We will repeatedly use the following basic fact about covering spaces: if $C$ is any subset of $X$, then $p|_{p^{-1}(C)}$ is a covering space.
\begin{lemma}\label{bigons}
Let $\widetilde{B}$ be an innermost bigon bounded by $\widetilde{\alpha}$ and $\widetilde{\beta}$  in $\widetilde{X}$.  Then $p(\widetilde{B})$ is an innermost bigon bounded by $\alpha$ and $\beta$ in $\bar{X}$.  Additionally, $\widetilde{B}$ does not contain any ramified points.
\end{lemma}
 Lemma \ref{bigons} implies in particular that if a cover has NU then innermost bigons in $\widetilde{X}$ do not contain preimages of branch points.
\begin{proof}
We first show that $p|_{\partial\widetilde{B}}$ is injective.  We deal with vertices, then edges.
Let $\widetilde y$ and $\widetilde{z}$ be the two vertices of $\widetilde{B}$ and let $y=p(\widetilde y)$ and $z=p(\widetilde z)$.  Because $\widetilde{\alpha}$ and $\widetilde{\beta}$ have opposite signs of intersection at $\widetilde{y}$ and $\widetilde{z}$ (after choosing arbitrary orientations of $\alpha$ and $\beta$), $y$ and $z$ are distinct points.  

Now we consider edges of $p|_{\partial\widetilde{B}}$.  Since $p$ is a covering space, $p|_{\widetilde{B}_{\widetilde{\alpha}}}$ is an immersion.  Since $\widetilde{B}$ is innermost, the interior of $\widetilde{B}_{\widetilde{\alpha}}$ maps to $\alpha\setminus\beta$.  And since $y\neq z$, $p|_{\widetilde{B}_{\widetilde{\alpha}}}$ is injective.  By symmetry, $p|_{\widetilde{B}_{\widetilde{\beta}}}$ is injective.

We now show that $p(\widetilde{B})$ is an innermost bigon in $X$.  Since $p|_{\partial\widetilde{B}}$ is injective, $p(\partial\widetilde{B})$ is a simple closed curve formed from one arc of $\alpha$ and one arc of $\beta$.  Moreover the arc of $\alpha$ only intersects $\beta$ at its endpoints and vice versa.  The image of the disk $\widetilde{B}$ under $p$ defines a nullhomotopy of $p(\partial\widetilde{B})$.  Also, the interior $p($int$\widetilde{B})$ is disjoint from $\alpha\cup\beta$.  It follows that $p(\partial\widetilde{B})$ is the boundary of an innermost bigon $B$ bounded by $\alpha\cup\beta$ and that $B=p(\widetilde{B})$.

It remains to show that that $\widetilde{B}$ does not contain any ramified points.  First we show that $p'=p|_{\widetilde{B}}$ is a covering space of $B=p(\widetilde{B})$.  
It suffices to show $\widetilde{B}$ is a connected component of $p^{-1}(B)$.  The interior of $B$ lies only to one side of $\partial B$ and is disjoint from $\alpha\cup\beta$, therefore $p^{-1}(p(\widetilde{B}))\setminus(\widetilde{\alpha}\cup\widetilde{\beta})$ lies only to one side of $\partial\widetilde{B}$ and is disjoint from $\widetilde{\alpha}\cup\widetilde{\beta}$.  It follows that $\widetilde{B}$ is the entire connected component of $p^{-1}(B)=p^{-1}(p(\widetilde{B}))$ containing $\widetilde{B}$.

The restriction of $p'$ to $\partial\widetilde{B}$ is a covering space of $\partial B$.  Because $p|_{\partial\widetilde{B}}=p'|_{\partial\widetilde{B}}$ is injective, the degree of $p'$ is 1. Since every nontrivially branched cover has degree greater than 1, 
$p|_{\widetilde{B}}$ is unbranched, hence $\widetilde{B}$ does not contain any ramified points.
 \end{proof}

 \begin{lemma}\label{annuli} 
  If $\widetilde{\alpha}$ and $\widetilde{\beta}$ are isotopic and disjoint in $\widetilde{X}$, then $\alpha$ and $\beta$ are isotopic in $X$.   \end{lemma}
 \begin{proof}
Any isotopy between $\widetilde{\alpha}$ and $\widetilde{\beta}$ in $\widetilde{X}$ projects to a homotopy between $\alpha$ and $\beta$ in $\bar{X}$.  Since $\alpha$ and $\beta$ are homotopic in $\bar{X}$ and disjoint, they bound an annulus in $\bar X$.  Therefore they bound an annulus $A$ in $X$, though $A$ might contain branch points. Our goal is to show that $A$ does not contain any branch points.

Because $A$ is an annulus in $X$, $\alpha$ and $\beta$ divide $X$ into two closed subsurfaces, $A$ and $Z$ with  $\partial A=\partial Z=\alpha\cup\beta$ sharing both boundary components $\alpha$ and $\beta$.
Let $\widetilde{A}=p^{-1}(A)$ and $\widetilde{Z}=p^{-1}(Z)$.

\vspace{6pt}
{\it Claim:} Either $\widetilde{A}$ or $\widetilde{Z}$ is a disjoint union of annuli in $\widetilde{X}$.

\vspace{6pt}

Assuming the claim, suppose $\widetilde{A}$ is a disjoint union of annuli.  The restriction of $p$ to $\widetilde{A}$ is a (possibly branched) covering space.  The Euler characteristic $\chi(\widetilde{A})$ is 0.  By the multiplicativity of Euler characteristic, $\chi(p(\widetilde{A}))=\chi(A)$ is also 0.  By the Riemann--Hurwitz formula \cite[Section 7.2.2]{primer}, an annulus with branch points has negative Euler characteristic, therefore $A$ has no branch points.

\vspace{9pt}
{\it Proof of claim}.
We will call the components of $\widetilde{A}$ and $\widetilde{Z}$ regions of $\widetilde{X}$.  Each region in $\widetilde{X}$ has at least one boundary component that is a component of $\widetilde{\alpha}$ and at least one boundary component that is a component of $\widetilde{\beta}$.  Similarly, each component of $\widetilde{\alpha}\cup\widetilde{\beta}$ is a boundary component of exactly one $\widetilde{A}$-region and exactly one $\widetilde{Z}$-region.  Because $\widetilde{\alpha}$ and $\widetilde{\beta}$ are isotopic, either the $\widetilde{A}$-region or $\widetilde{Z}$-region bounded by any component of $\widetilde{\alpha}\cup\widetilde{\beta}$ is an annulus.  

Suppose for every component of $\widetilde{\alpha}\cup\widetilde{\beta}$, exactly one of the adjacent regions is an annulus.  Then the regions of $X$ alternate between annular and non-annular regions of $X$.  Therefore either all of the components of $\widetilde{A}$ or all of the components of $\widetilde{Z}$ are annuli and the claim is proved.

Suppose now there is some component of $\widetilde{\alpha}\cup\widetilde{\beta}$ such that both the adjacent regions are annuli.  We say two such annular regions are {\it adjacent}.  The adjacency of annular regions generates an equivalence relation of annular regions of $\widetilde{X}$.  Let $\{\mathcal{F}_i\}$ be the set of equivalence classes of annular regions of $\widetilde{X}$.  Where $|\mathcal{F}_i|>1$, there exist two annuli in each $\mathcal{F}_i$ that are adjacent to only one annulus in $\mathcal{F}_i$; call these {\it end annuli}.  Where $|\mathcal{F}_i|=1$, $\mathcal{F}_i$ consists of a single annulus, which is also an end annulus.

Because $\widetilde{\alpha}$ and $\widetilde{\beta}$ are isotopic in $\widetilde{X}$, there are the same number of components of $\widetilde{\alpha}$ and $\widetilde{\beta}$ in the union of the elements of $\mathcal{F}_i$ for any $i$.  It follows that $|\mathcal{F}_i|$ is odd.  If the end annuli of some $\mathcal{F}_i$ are in $\widetilde{A}$, then the adjacent non-annular regions are in $\widetilde{Z}$.  All annular regions in $\widetilde{X}$ are in some class $\mathcal{F}_i$.  By replacing annular regions in the case above with the $\mathcal{F}_i$, we see that the $\mathcal{F}_i$ alternate with the non-annular regions in $\widetilde{Z}$.  Then end annuli of all $\mathcal{F}_i$ are in $\widetilde{A}$, and the claim is proved.
\end{proof}

We may now prove Theorem \ref{ramified}.

\begin{proof}[Proof of Theorem \ref{ramified}]
By Proposition \ref{curveisotopy}, it is enough to show that $p$ has the isotopy projection property.
	Let $\alpha$ and $\beta$ be transverse simple closed curves in $X$.  We need to show that if $\widetilde{\alpha}=p^{-1}(\alpha)$ and $\widetilde{\beta}=p^{-1}(\beta)$ are isotopic, then $\alpha$ and $\beta$ are isotopic.  If $\widetilde{\alpha}$ and $\widetilde{\beta}$ are not disjoint, $\widetilde{\alpha}$ and $\widetilde{\beta}$ bound an innermost bigon $\widetilde{B}$.  By assumption, there are no preimages of branch points with unramified preimage.  Therefore Lemma \ref{bigons} says that $p(\widetilde{B})$ is an innermost bigon in $\alpha\cup\beta$ in $X$ and it does not contain any branch points.  There is then an isotopy of $\alpha$ in $X$ that reduces the number of intersections of $\alpha$ and $\beta$ (by passing $\alpha$ through $p(\widetilde{B})$).  We continue reducing intersections of $\alpha$ and $\beta$ by isotopy of $\alpha$ in $X$ until $\alpha$ and $\beta$ are disjoint, hence $\widetilde{\alpha}$ and $\widetilde{\beta}$ are also disjoint.	 
Once $\widetilde{\alpha}$ and $\widetilde{\beta}$ are disjoint, isotopic multicurves, Lemma \ref{annuli} implies that $\alpha$ and $\beta$ are isotopic.
\end{proof}

Theorem \ref{ramified} can also be proved using quadratic differentials.  We sketch the proof.  Because the branch points are treated as marked points, a quadratic differential on $X$ may have a pole at a branch point.  However, quadratic differentials on $\widetilde{X}$ are not allowed to have any poles.  When the preimages of all branch points are ramified, a quadratic differential on $X$ lifts to a quadratic differential on $\widetilde{X}$ \cite[Section 11.3.3]{primer}.  A given surface $S$ (and hyperbolic structure) and a quadratic differential on $S$ determine a Teichm\"{u}ller geodesic in the Teichm\"uller space of $S$.  Since all quadratic differentials on $X$ lift to quadratic differentials on $\widetilde{X}$, Teichm\"uller geodesics in Teichm\"uller space of $X$ lift to Teichm\"uller geodesics in the Teichm\"uller space of $\widetilde{X}$.  Therefore the Teichm\"uller space of $X$ isometerically embeds in the Teichm\"uller space of $\widetilde{X}$, which is equivalent to the Birman--Hilden property by the work of MacLachlan and Harvey \cite{MH}.

\section{Simple covers}\label{Simple}
A {\it simple cover} of surfaces is a covering space such that each branch point has one preimage with ramification number 2 and all other preimages are unramified.  If the degree of a simple cover is $n$, then each branch point has $n-1$ preimages.

\subsection{Colorings} We first develop a consistent way of labeling preimages of branch points.  The ideas in this section are inspired by the work of Berstein and Edmonds \cite{BE}.

Let $\widetilde X\rightarrow X$ be an $n$-fold branched covering space of surfaces and fix some orientation of $X$.  Fix a base point in $X$ that is not a branch point; call it $x_0$.   
Label the points in $p^{-1}(x_0)$ with the numbers $1,\cdots,n$ and fix throughout.  Let $x_1,\cdots,x_k$ be the branch points of the covering space.  Fix a set of simple arcs $\gamma_1,\cdots,\gamma_k$ in $X$ that are disjoint on their interiors and such that $\gamma_i$ connects $x_0$ and $x_i$. We call such a set of arcs an {\it arc system}.  
Let $\rho:\pi_1(X^{\circ},x_0)\rightarrow\Sigma_n$ be the monodromy homomorphism of the unbranched covering space $\widetilde{X}^{\circ}\rightarrow X^{\circ}$ with respect to our labeling of the points in $p^{-1}(x_0)$.  Let $c_i\in\pi_1(X^{\circ},x_0)$ be the homotopy class of the counterclockwise-oriented boundary of a neighborhood of $\gamma_i\subset X$.  We define the {\it coloring} of $x_i$, denoted $\color(x_i)$, as $\color(x_i)=\rho(c_i)$.   

\begin{figure}[h]
\begin{center}
\labellist\small\hair 2.5pt
        \pinlabel {$x_i$} by 1 0 at 134 45
         \pinlabel {$1$} by 0 0 at 25 110
          \pinlabel {$\widetilde{x}_i^1$} by 0 0 at 35 193
          \pinlabel {$2$} by 0 0 at 26 238
 \pinlabel {$(12)$} by 0 0 at 55 220
           \pinlabel {$3$} by 0 0 at 77 150
            \pinlabel {$5$} by 0 0 at 123 236
            \pinlabel {$4$} by 0 0 at 165 150
             \pinlabel {$6$} by 0 0 at 230 125
 \pinlabel {$(345)$} by 1 0 at 168 220
  \pinlabel {$\widetilde{x}_i^2$} by 0 0 at 131 193
          \pinlabel {$8$} by 0 0 at 230 235
           \pinlabel {$7$} by 0 1 at 285 188
             \pinlabel {$\widetilde{x}_i^3$} by 0 0 at 240 193
          \pinlabel {$9$} by 0 1 at 178 188
 \pinlabel {$(6789)$} by 0 0 at 265 220
        \pinlabel {$c_x$} by 0 1 at 115 65
          \pinlabel $\gamma_x$ by 0 -1 at 115 25 
         \pinlabel $x_0$ by 1 -1 at 95 0 
      
        \endlabellist
\centerline{\includegraphics[scale=1]{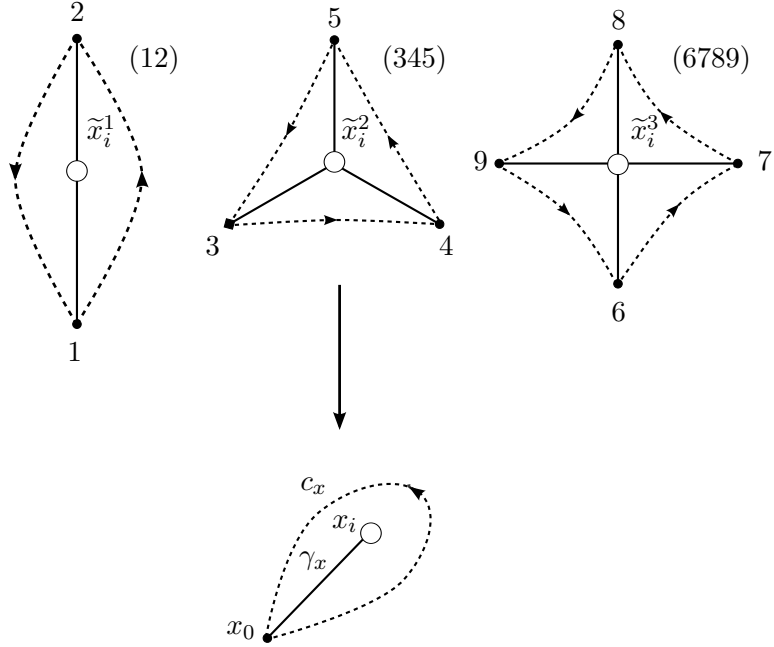}}
\end{center}
\caption{The coloring $(12)(345)(6789)$.}

\label{coloringcases}\end{figure}
Alternatively, we can think of coloring the preimages of the branch point $x_i$ instead of coloring $x_i$ itself.  Let $x_i^1,\cdots,x_i^{\ell_i}$ be the preimages of $x_i$ under $p$.  Each arc of $p^{-1}(\gamma_i)$ connects some point of $p^{-1}(x_0)$ to some $x_i^r$.  The arcs of $p^{-1}(\gamma_i)$ containing a given $x_i^r$ can be ordered cyclically counterclockwise near $x_i^r$.  This defines a cyclic ordering of the points of $p^{-1}(x_0)$ to which $p^{-1}(\gamma_i)$ connects $x_i^r$.  Define $\color(x_i^r)$ as the cycle in $\Sigma_n$ that cyclically permutes the labels of the points; see Figure \ref{coloringcases}.  Then it is easy to see that $$\color(x_i)=\prod_{r=1}^{\ell_i}\color(x_i^r).$$

\p{Coloring graph}  Let $\widetilde{X}\rightarrow X$ be a branched covering space and $x_i$ and $x_j$ be branch points in $X$.  

Let $x_i^1,\cdots,x_i^{\ell_i}$ in $\widetilde{X}$ be the elements of $p^{-1}(x_i)$ and $x_j^1,\cdots,x_j^{\ell_j}$ the elements of $p^{-1}(x_j)$.  We form a loopless bipartite multigraph $\Gamma_{ij}$ as follows: there is a vertex for every $x_i^r$ and $x_j^q$.    
  There is one edge of $\Gamma_{ij}$ between the vertices corresponding to $x_i^r$ and $x_j^q$ for every element of the set $\{1,\cdots,n\}$ in the support of both $\color(x_i^r)$ and $\color(x_j^q)$ (the support of a permutation $\sigma\in\Sigma_n$ is any number $1,\cdots,n$ such that $\sigma\cdot n\neq n$).  We note that the degree of the vertex of $\Gamma_{ij}$ corresponding to $x_i^{\ell}$ is the ramification number of $x_i^{\ell}$.

We now give an alternative description of the edges of the vertices of $\Gamma_{ij}$.  Let $\gamma=\gamma_i\cup\gamma_j$.  
As above, the are the elements of $x_i^r$ and $x_j^q$ and there is one edge between the vertices $x_i^r$ and $x_j^q$  for each arc of $p^{-1}(\gamma)$ connecting $x_i^r$ and $x_j^q$.

\subsection{Simple covers do not have the Birman--Hilden property}\label{covers}
Before we prove Theorem \ref{simple}, we need a lemma.

\begin{lemma}\label{goodgamma}
Let $p:\widetilde{X}\rightarrow X$ be a $n$-fold cover of surfaces where $n\geq3$, $\chi(\widetilde{X})<0$, and there are at least two branch points $x_1,...,x_k$ in $X$.  Choose a base point $x_0$ that is not a branch point.  There exists an arc system $\gamma_1,\cdots,\gamma_k$ in $X$ where with respect to $\gamma_1,\cdots,\gamma_k$, $\color(x_i)\neq\color(x_j)$ for some $1\leq i,j\leq k$.
\end{lemma}
\begin{proof}
First choose any arc system $\{\gamma_1,\cdots,\gamma_k\}$.  
As above, let $c_i\in\pi_1(X^{\circ},x_0)$ be the homotopy class of the boundary of a regular neighborhood of $\gamma_i\subset X$ that is oriented counterclockwise.  To prove the lemma we consider two cases: when $X^{\circ}$ is a punctured sphere and when $X^{\circ}$ is a punctured surface of genus $g$ with $g\geq 1$.  In both cases since the cover is simple $\color(x_i)$ is a transposition for all $i$.

Suppose $X^{\circ}$ is a punctured sphere.  In this case $\{c_1,\cdots,c_k\}$ generate $\pi_1(X^{\circ},x_0)$.  As $\rho(c_i)=\color(x_i)$, the cycles $\{\color(x_1),\cdots,\color(x_k)\}$ generate $\rho(\pi_1(X^{\circ},x_0))$.  Since $n\geq3$, and $\widetilde{X}$ is connected, $\rho(\pi_1(X^{\circ},x_0))$ is generated by more than one transposition.  Hence there exist $x_i$ and $x_j$ such that $\color(x_i)\neq\color(x_j)$.

Now suppose $X$ has positive genus.  We will assume that with respect to the arc system $\gamma_1,\cdots,\gamma_k$, $\color(x_i)=\color(x_j)$  for all $i$ and $j$.  (Otherwise we have nothing to show).  Without loss of generality $\color(x_i)=(12)$.  There exists an embedded disk $D\subset X$ with $x_0\in\partial D$ and $\gamma_i\subset D$ for all $i$ as in Figure \ref{notequal}.  
\begin{figure}[h]
\begin{center}
\labellist\small\hair 2.5pt
        \pinlabel {$D$} by 1 0 at 335 45
         \pinlabel {$\gamma_1$} by 0 1 at 355 100
 \pinlabel {$\gamma_k$} by 0 -1 at 355 48
 \pinlabel {$\delta$} by 0 0 at 286 41
  \pinlabel {$\gamma_1'$} by 0 -1 at 250 120
         \pinlabel $x_0$ by 1 -1 at 328 81 
      
        \endlabellist
\centerline{\includegraphics[scale=1]{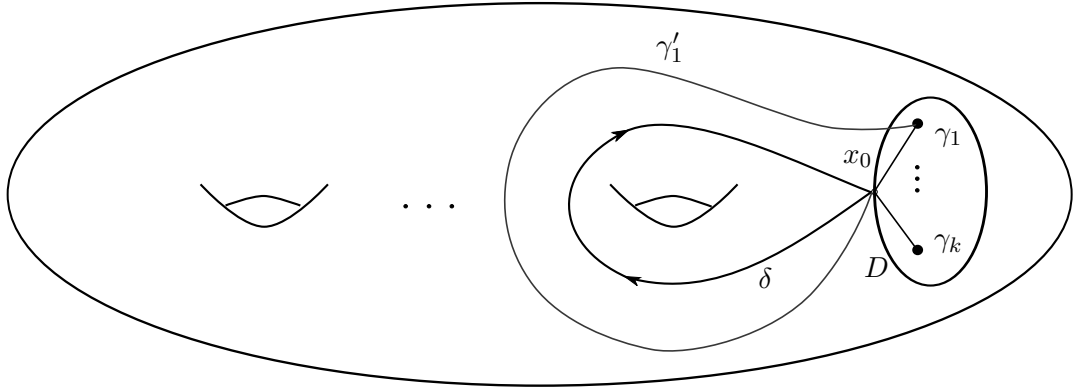}}
\end{center}
\caption{The picture used in the proof of Lemma \ref{goodgamma}.}
\label{notequal}
\end{figure}

Choose a nonseparating simple loop $\delta$ in $X^{\circ}$ that is disjoint from the interior of $D$ and where $\rho(\delta)\cdot1\not\in\{1,2\}$. Such $\delta$ exists because $\pi_1(X^{\circ},x_0)$ is generated by the $c_i$ and nonseparating simple loops in $X^{\circ}$ disjoint from the interior of $D$ and because $\rho(\pi_1(X^{\circ},x_0))$ is transitive.  Up to homeomorphism of $X$ and relabeling the $c_i$, we have the configuration shown in Figure \ref{notequal}. The conjugate $\delta c_1\delta^{-1}$ is an element of $\pi_1(X^{\circ},x_0)$ that collapses in $X$ to the arc $\gamma_1'$ shown in Figure \ref{notequal}.  The arc $\gamma_1'$ is disjoint from the interior of each $\gamma_i,i\neq1$.  We also have that $\rho(\delta c_1\delta^{-1})=(\rho(\delta)\cdot 1 \rho(\delta)\cdot 2)\neq(12)$.  Therefore with respect to $\{\gamma_1',\gamma_2,\cdots,\gamma_k\}$, $\color(x_1)$ is not equal to $\color(x_i), i\neq 1$.
\end{proof}
We now prove Theorem \ref{simple}, that simple covers do not have the Birman--Hilden property when the degree of the cover is at least 3.  
\begin{proof}[Proof of Theorem \ref{simple}]
%
Recall that by hypothesis $\widetilde{X}\rightarrow X$ has at least two branch points and that $X^{\circ}$ is not a pair of pants. 

Label the branch points $x_1,\cdots,x_k$.  Fix an arc system $\gamma_1,\cdots,\gamma_k$ so that there exist branch points $x_i$ and $x_j$ with $\color(x_i)\neq\color(x_j)$.  This is possible by Lemma \ref{goodgamma}.

Because there are only two vertices of $\Gamma_{ij}$ with degree larger than one, the only circuits in $\Gamma_{ij}$ can occur when there are two edges connecting the vertices of degree 2.  Since $\color(x_i)\neq\color(x_j)$, this is not the case.  It follows that $\Gamma_{ij}$ is a forest.

Let $\alpha$ be a simple closed curve in $X$ homotopic to the boundary of the neighborhood of $\gamma_i\cup\gamma_j$.  Then $p^{-1}(\alpha)$ is homotopic to the boundary of a neighborhood of $p^{-1}(\gamma_i\cup\gamma_j)$.  The simple closed curve $\alpha$ is essential because it bounds a disk containing two branch points and $X^{\circ}$ is not a sphere with three punctures.  Since $\Gamma_{ij}$ is a forest, all components of $p^{-1}(\alpha)$ are trivial.  Therefore $p$ does not have the weak curve lifting property, and by Proposition \ref{necessity} it does not have the Birman--Hilden property.
\end{proof}

\bibliography{SIIC}
\end{document}